\documentclass[11pt]{amsart}

\usepackage{amsmath,amssymb,amsthm,mathtools}
\usepackage[margin=1in]{geometry}
\usepackage[
  colorlinks=true,
  linkcolor=blue,
  citecolor=blue,
  urlcolor=blue
]{hyperref}

\newtheorem{theorem}{Theorem}[section]
\newtheorem{proposition}[theorem]{Proposition}
\newtheorem{lemma}[theorem]{Lemma}
\newtheorem{corollary}[theorem]{Corollary}

\theoremstyle{remark}

\newcommand{\Z}{\mathbb{Z}}
\newcommand{\F}{\mathbb{F}}

\newcommand{\one}{\mathbf{1}}

\DeclareMathOperator{\lc}{lc}
\DeclareMathOperator{\rank}{rank}

\title[Kissing Numbers $\tau(\mathcal{L}_n)\geq e^{2\sqrt{n}}$]
  {A Lattice Family with Kissing Numbers
   $\tau(\mathcal{L}_n)\geq e^{2\sqrt{n}}$}

\author{Thijs Laarhoven}
\address{NXP Semiconductors}
\email{mail@thijs.com}

\author{Scott Duke Kominers}
\address{Harvard Business School; Department of Economics and Center of
  Mathematical Sciences and Applications, Harvard University; and
  a16z crypto}
\email{kominers@fas.harvard.edu}

\date{\today}

\subjclass[2020]{11H31, 52C17, 11T06}
\keywords{lattice, kissing number, finite field, polynomial congruence,
  Reed-Solomon code}

\thanks{%
  We used LLMs to assist with analysis and synthesis in the preparation of
  this article, particularly GPT-6 Astra and Claude Fable 5.1 (accessed in
  part via Poe with the support of Quora, where Kominers is an advisor).
  The initial construction with
  $\tau(\mathcal{L}_n)\geq e^{\sqrt{n}}$ was developed by the first
  author~\cite{Laarhoven} with GPT-6 Astra, and the improved construction
  with $\tau(\mathcal{L}_n)\geq e^{2\sqrt{n}}$ was developed by the second
  author~\cite{Kominers} in dialogue with GPT-6 Astra. This paper combines,
  subsumes, replaces, and extends the two authors' respective earlier
  preprints~\cite{Laarhoven} and~\cite{Kominers}.
  The problem, methods, and final written form are our own; and of course
  any errors remain our responsibility. This work was conducted while
  Kominers was visiting the Technological Innovation, Entrepreneurship,
  and Strategic Management (TIES) Group at the MIT Sloan School of
  Management; he greatly appreciates their hospitality.%
}

\begin{document}

\begin{abstract}
For all prime powers $q\geq5$, we construct lattices
$\mathcal{L}_q\subseteq\Z^q$ with kissing numbers
\[
  \tau(\mathcal{L}_q)\geq
  \left(\frac{1}{2\pi e^2}+o(1)\right)\sqrt{q}\,e^{2\sqrt{q}}.
\]
The same asymptotic bound holds on a set of integer dimensions of natural
density $1$, and in every sufficiently large integer dimension $n$ with an
additional factor $e^{-\tfrac{1}{2}n^{1/40}}$. The construction is an
extension of a previous construction by Bennett--Peikert based on
Reed--Solomon codes.
\end{abstract}

\maketitle


\section{Introduction}
\label{sec:introduction}

For a nonzero Euclidean lattice $\mathcal{L}$, we write
\[
  \min \mathcal{L}
  =\min_{0\ne x\in\mathcal{L}}\langle x,x\rangle
  \quad\text{and}\quad
  \tau(\mathcal{L})
  =\bigl|\{x\in\mathcal{L}:
    \langle x,x\rangle=\min \mathcal{L}\}\bigr|
\]
for the \textit{(squared) minimum} and \textit{kissing number},
respectively. Note that both signs of each minimal vector are counted
in $\tau$.

The Barnes--Wall lattices~\cite{BarnesWall} in dimensions $n=2^m$ have
kissing numbers
\[
  \prod_{i=1}^{m}(2^i+2)=\exp\bigl(\Theta((\log n)^2)\bigr);
\]
see Nebe~\cite[Lemma~3.4]{Nebe}.
Although individual lattices can have larger kissing numbers than the
Barnes--Wall lattice in a given dimension,\footnote{%
  For example, the Mordell--Weil lattice of rank $128$ has kissing number
  more than $170$ times that of $\mathrm{BW}_{128}$~\cite{Elkies}.%
}
to the best of our knowledge, no previously known lattice family achieves
growth beyond this quasipolynomial scale (see also~\cite{Alon,BGS}).

In this note, we give a construction of a lattice family with much larger
kissing numbers, extending a Reed--Solomon lattice construction of Bennett
and Peikert~\cite{BennettPeikert}. In particular, we obtain kissing numbers
at least $e^{2\sqrt{n}+o(\sqrt{n})}$ through prime powers.

\begin{theorem}
\label{thm:main}
There are full-rank lattices $\mathcal{L}_q\subseteq\Z^q$, indexed by the
prime powers $q\geq5$, such that, as $q\to\infty$ through prime powers,
\begin{equation}
  \label{eq:main}
  \tau(\mathcal{L}_q)\geq
  \left(\frac{1}{2\pi e^2}+o(1)\right)\sqrt{q}\,e^{2\sqrt{q}}.
\end{equation}
In particular, $\tau(\mathcal{L}_q)\geq e^{2\sqrt{q}}$ for all sufficiently
large prime powers $q$.
\end{theorem}

The same asymptotic lower bound holds on a set of integer dimensions of
natural density $1$. In every sufficiently large dimension $n$, we also
obtain the bound~\eqref{eq:all-dimensions}, with exponent
$2\sqrt{n}-\tfrac{1}{2}n^{1/40}$; see
Section~\ref{sec:other-dimensions}.

Bennett, Golovnev, and Stephens-Davidowitz~\cite{BGS} identify lattice
families with kissing numbers at least $2^{n^\varepsilon}$ as already of
interest for the complexity of lattice problems, for any fixed
$\varepsilon>0$. Our construction supplies such a family with
$\varepsilon=1/2$.


\section{The construction for prime power dimensions}
\label{sec:construction}

To obtain a large kissing number, we need to assemble many vectors of the
same length into a lattice in which they are all shortest. Correspondingly,
the real challenge is to ensure that integer combinations of the candidate
minimal vectors do not end up having shorter lengths.

Fix an integer $a\geq2$ with $2a<q$. The candidates for our construction are
\[
  z_{A,B}=\one_A-\one_B,
  \qquad
  A,B\subseteq\F_q,\quad |A|=|B|=a,\quad A\cap B=\varnothing,
\]
where $\one_A$ denotes the coordinate indicator vector of $A$. These
vectors lie in the hyperplane of coordinate sum zero, on the sphere of
radius $\sqrt{2a}$ centered at the origin. We seek a sublattice containing
many of these points, yet containing no nonzero point inside the sphere.
There are
\[
  \binom{q}{a}\binom{q-a}{a}
  =\binom{q}{2a}\binom{2a}{a}
\]
ordered pairs: after choosing the support, we choose which $a$ coordinates
are positive.

The sum-zero condition suggests a way to measure---and thus
control---shortness. For an integer vector $z$ in this hyperplane, the
positive coordinates and the absolute values of the negative coordinates
have the same sum, say $m$. Since $z_s^2\geq|z_s|$ for integer coordinates,
\[
  \|z\|^2\geq\sum_s|z_s|=2m.
\]
Equality holds precisely when every coordinate is $0$, $1$, or $-1$. Our
candidates attain equality with $m=a$. It therefore suffices to find a
lattice condition that retains many of these candidates and forces
$m\geq a$ for every nonzero vector.

We encode the positive and negative coordinates as root multiplicities in
two polynomials. Put
\[
  P_Z(X)=\prod_{s\in Z}(X-s).
\]
The two polynomials $P_A$ and $P_B$ are monic of degree $a$, so their
difference has degree at most $a-1$. We call a pair $A,B$ \emph{good} when
the degree of $P_A-P_B$ is exactly $a-1$, and let $\mathcal{G}$ be the set
of good pairs. Each good pair determines a monic polynomial by the
assignment
\begin{equation}
  \label{eq:assignment-map}
  (A,B)\longmapsto
  f_{A,B}=\frac{P_A-P_B}{\lc(P_A-P_B)},
\end{equation}
where $\lc$ denotes the leading coefficient. In particular,
$P_A\equiv P_B\pmod{f_{A,B}}$ by construction.

There are only $q^{a-1}$ monic polynomials of degree $a-1$. The
assignment~\eqref{eq:assignment-map} therefore partitions $\mathcal{G}$
into that many classes, some possibly empty. Write $R_f$ for the number
of good pairs assigned to $f$. By the pigeonhole principle, some class
has size at least $|\mathcal{G}|/q^{a-1}$. We select a modulus $f$
indexing such a class. In the proof below, we show that the degree
condition discards at most a $1/(q-2a+1)$ fraction of all ordered disjoint
pairs.

For the chosen modulus, the congruence has a lattice interpretation.
Retain the coordinates $s$ with $f(s)\ne0$, so that $X-s$ is invertible
modulo $f$. Let $K_f$ consist of the integer vectors on these coordinates
satisfying
\[
  \sum_s z_s=0,
  \qquad
  \prod_s(X-s)^{z_s}=1
  \quad\text{in }(\F_q[X]/(f))^\times.
\]
Every good pair assigned to $f$ gives a vector $z_{A,B}\in K_f$. Indeed,
none of its nonzero coordinates is removed, and its polynomial congruence
is precisely the multiplicative relation above.

The degree $a-1$ is the largest degree available for our candidate
differences. For any nonzero vector in $K_f$, the positive and negative
coordinates give distinct monic polynomials of degree $m$---and those
polynomials' difference is a nonzero multiple of $f$. The leading terms
cancel, so $a-1\leq m-1$, as required. Thus, we see that $K_f$ has no
nonzero points inside the sphere. The assigned candidates remain on its
boundary, and Lemma~\ref{lem:minimum} shows that they are exactly the
shortest vectors. We complete the rank by adjoining long orthogonal
directions, which leave these boundary points unchanged. Finally, we
choose $a$ near $\sqrt{q}$ to optimize their number.


\subsection*{Relation to earlier constructions}

As mentioned above, our approach builds on an earlier construction of
Bennett and Peikert~\cite{BennettPeikert}. The unrestricted multiplicative
kernel we use is a polynomial lattice of Li, Ling, Xing, and
Yeo~\cite[Section~3, Lemma~2]{LLXY}. For a modulus of degree $d$, Li, Ling,
Xing, and Yeo prove a squared-minimum bound of $d$, with equality only
when all nonzero coordinates have the same sign; their work concerns
closest-vector problems and cryptographic applications and does not give
a kissing-number count.
Our sum-zero restriction gives the stronger squared-minimum bound $2a$
when $d=a-1$. Beyond constructing the family, our argument counts the
signed vectors attaining this bound and selects a modulus by averaging.

The construction also connects to several classical examples. For prime
$q$, the analogous construction on all $q$ coordinate labels, with the
finite-modulus relation replaced by
\[
  \prod_{s\in\F_q}(1-sT)^{z_s}\equiv1\pmod{T^a},
\]
gives Craig's lattice $A_{q-1}^{(a)}$~\cite{Craig,ConwaySloane}.
We return to this congruence at infinity, and to kissing-number bounds for
Craig's lattices, in Section~\ref{sec:remarks}.
For irreducible $f$, the multiplicative relation defining $K_f$ uses the
linear factors underlying the Bose--Chowla $B_h$-sets~\cite{BoseChowla};
for split squarefree $f$, it has a discrete-logarithm interpretation as in
the work of Ducas and Pierrot~\cite{DucasPierrot}.

Selecting a large residue class of constant-weight words also appears in
the bounds of Graham and Sloane~\cite{GrahamSloane}.
Derksen~\cite[Theorem~5(a), Lemma~6]{Derksen} uses the same unit group and
linear factors to construct $B_h$-sequences for constant-weight codes, and
chooses the modulus to minimize the order of the unit group and thereby
improve the lower bounds on code sizes. Our modulus, by contrast, is
chosen to enable many minimal vectors.

For finite non-lattice packings, Alon~\cite{Alon} obtains more than
$2^{\sqrt{n}}$ contacts at every ball when $n$ is a proper power of $4$.
The existence of lattice families with exponentially large kissing
numbers remains open; the claimed constructions of
Vl\u{a}du\c{t}~\cite{Vladut} were found to be invalid~\cite{BGS}.


\section{Proof of the kissing bound}
\label{sec:kissing-bound}

We first make the construction of Section~\ref{sec:construction}
quantitative by proving an explicit lower bound on the kissing numbers
attained by the resulting lattices. For the bound, we allow the
coordinates to be indexed by any $n$-element subset of $\F_q$. We then
deduce Theorem~\ref{thm:main} by taking $n=q$, choosing
$a=\lfloor\sqrt{q}\rfloor$, and evaluating the bound asymptotically. The
flexibility to take $n<q$ will be useful when extending the construction
to general dimensions in Section~\ref{sec:other-dimensions}.

\begin{proposition}
\label{prop:finite}
Let $q$ be a prime power, and let $a,n$ be integers with $2\leq a$ and
$2a<n\leq q$. There is a full-rank lattice $\mathcal{L}\subseteq\Z^n$ such
that $\min \mathcal{L}=2a$ and
\begin{equation}
  \label{eq:finite}
  \tau(\mathcal{L})\geq
  \left(1-\frac{1}{n-2a+1}\right)
  \frac{\binom{n}{a}\binom{n-a}{a}}{q^{a-1}}.
\end{equation}
\end{proposition}

To prove Proposition~\ref{prop:finite}, we first count pairs and select a
modulus. We then explicitly construct the lattice for the chosen modulus
via the recipe described in Section~\ref{sec:construction}. After that,
we prove the minimum bound, and complete the lattice to full rank.


\subsection*{Assigning pairs to moduli}

Fix $T\subseteq\F_q$ with $|T|=n$. For disjoint $a$-element subsets
$A,B\subseteq T$, put $z_{A,B}=\one_A-\one_B$. These are distinct integer
vectors of squared norm $2a$, one for each of the
$\binom{n}{a}\binom{n-a}{a}$ ordered pairs.

As in Section~\ref{sec:construction}, put
$P_Z(X)=\prod_{s\in Z}(X-s)$. The leading terms of $P_A$ and $P_B$ cancel,
so $\deg(P_A-P_B)\leq a-1$, and the coefficient of $X^{a-1}$ is
\begin{equation}
  \label{eq:leading-coefficient}
  \sum_{s\in B}s-\sum_{s\in A}s.
\end{equation}
Let $\mathcal{G}$ consist of the pairs for which the
coefficient~\eqref{eq:leading-coefficient} is nonzero; these are exactly
the \textit{good pairs} defined in Section~\ref{sec:construction}, with
$\deg(P_A-P_B)=a-1$.

For each monic polynomial $f$ of degree $a-1$, let $R_f$ be the number of
good pairs assigned to $f$ by the reduction~\eqref{eq:assignment-map}.
Equivalently, $R_f$ counts the ordered disjoint pairs satisfying
\begin{equation}
  \label{eq:assigned}
  P_A-P_B=cf
  \qquad\text{for some }c\in\F_q^\times.
\end{equation}
Each good pair is assigned to exactly one modulus. Reversing the pair
leaves its normalized difference unchanged and gives the opposite
integer vector. Thus the ordered-pair convention counts both signs, in
every characteristic.


\subsection*{Choosing a large class}

We bound the number of good pairs by varying one coordinate. Order the
elements of $A$ and $B$, and fix all but the last element of $B$. There
are $n-2a+1$ available values for the last element. Each gives a different
value of the coefficient~\eqref{eq:leading-coefficient}, so at most one
makes it zero. Every ordered pair of subsets is represented by exactly
$(a!)^2$ such ordered lists. Dividing the good and total list counts by
this common factor gives
\begin{equation}
  \label{eq:good-pairs}
  |\mathcal{G}|\geq
  \left(1-\frac{1}{n-2a+1}\right)\binom{n}{a}\binom{n-a}{a}.
\end{equation}

The fibers of~\eqref{eq:assignment-map} partition $\mathcal{G}$, so
\[
  \sum_{\substack{f\text{ monic}\\ \deg f=a-1}}R_f=|\mathcal{G}|;
\]
there are $q^{a-1}$ terms in this sum, including any zero terms. We choose
a modulus $f$ for which
\begin{equation}
  \label{eq:chosen-modulus}
  R_f\geq\frac{|\mathcal{G}|}{q^{a-1}}
  \geq\left(1-\frac{1}{n-2a+1}\right)
  \frac{\binom{n}{a}\binom{n-a}{a}}{q^{a-1}}>0.
\end{equation}
It remains to realize these $R_f$ pairs as the minimal vectors of a
full-rank lattice. The construction and minimum estimate below apply to
every monic degree-$(a-1)$ modulus; we use the one just selected.


\subsection*{The lattice for the chosen modulus}

Negative exponents in a multiplicative relation require units. We
therefore retain only the labels
\[
  S_f=\{s\in T:f(s)\ne0\}.
\]
This removes no coordinate of any pair counted by $R_f$. Indeed, if
$s\in A$, then disjointness gives
\[
  cf(s)=(P_A-P_B)(s)=-P_B(s)\ne0,
\]
and the same argument applies to $s\in B$.

For $s\in S_f$, the residue class of $X-s$ is a unit in $\F_q[X]/(f)$.
Write $U_f=(\F_q[X]/(f))^\times$ and define
\begin{equation}
  \label{eq:kernel}
  K_f=\left\{
    z\in\Z^{S_f}:\sum_{s\in S_f}z_s=0,
    \ \prod_{s\in S_f}(X-s)^{z_s}=1\text{ in }U_f
  \right\}.
\end{equation}
The first condition in~\eqref{eq:kernel} balances the positive and
negative coordinates. The second admits every assigned vector: for a
pair counted by $R_f$,
\[
  \prod_{s\in S_f}(X-s)^{(z_{A,B})_s}
  =P_A P_B^{-1}=1
  \quad\text{in }U_f.
\]
Both conditions are preserved under addition and negation, so they
define an integer lattice. Moreover, $U_f$ is finite. Hence $K_f$ has
finite index in the integer sum-zero lattice on $S_f$, and rank
$|S_f|-1$.


\subsection*{Excluding shorter vectors}

We now show that every nonzero vector of $K_f$ has squared norm at least
$2a$. For this, we adapt the degree argument of~\cite[Section~3,
Lemma~2]{LLXY} to our sum-zero setting. The lower bound also follows from
Derksen's $B_{a-1}$-sequence construction~\cite[Theorem~5(a)]{Derksen}.
The argument below additionally identifies all equality vectors.

\begin{lemma}
\label{lem:minimum}
Every nonzero vector of $K_f$ has squared norm at least $2a$. Its vectors
of squared norm $2a$ are exactly the vectors $\one_A-\one_B$ arising from
the $R_f$ pairs in~\eqref{eq:assigned}.
\end{lemma}

\begin{proof}
Let $0\ne z\in K_f$. The sum of the positive coordinates equals the sum
of the absolute values of the negative coordinates; we write their common
value as $m$. Form the monic degree-$m$ polynomials
\[
  P_+(X)=\prod_{z_s>0}(X-s)^{z_s},
  \qquad
  P_-(X)=\prod_{z_s<0}(X-s)^{-z_s};
\]
their root supports are disjoint, so unique factorization gives
$P_+\ne P_-$. The multiplicative relation defining $K_f$ gives
$P_+\equiv P_-\pmod{f}$. Thus $P_+-P_-$ is a nonzero multiple of $f$.
Since its leading terms cancel,
\begin{equation}
  \label{eq:degree}
  a-1=\deg f\leq\deg(P_+-P_-)\leq m-1.
\end{equation}
It follows that $m\geq a$. Each coordinate is an integer, so
$z_s^2\geq|z_s|$. Therefore
\begin{equation}
  \label{eq:mass}
  \langle z,z\rangle
  =\sum_s z_s^2\geq\sum_s|z_s|=2m\geq2a.
\end{equation}

Equality in~\eqref{eq:mass} requires both $m=a$ and that every coordinate
lies in $\{0,1,-1\}$. Let $A$ and $B$ be the positive and negative
supports. They each have size $a$, and~\eqref{eq:degree} forces
$P_A-P_B=cf$ for some $c\ne0$. Thus $(A,B)$ is counted by $R_f$.
Conversely, every pair in~\eqref{eq:assigned} gives a vector of $K_f$
with squared norm $2a$.
\end{proof}

Note that no irreducibility or squarefreeness assumption on $f$ was used
in the proof of Lemma~\ref{lem:minimum}, and moreover the argument works
in every characteristic. For the selected modulus, $R_f>0$. Hence the
minimum is attained, and $K_f$ has squared minimum $2a$ and exactly $R_f$
minimal vectors.


\subsection*{Completing the rank}

The argument thus far shows that our constructed lattice $K_f$ has the
target number of minimal vectors. However, $K_f$ is not full-rank. Thus,
to complete the proof of Proposition~\ref{prop:finite}, we must show that
it is possible to extend $K_f$ to a full-rank lattice without introducing
any additional vectors of squared norm at most $2a$.

Put $N=|S_f|$ and $M=2a+1$. The degree bound on $f$ gives
$N\geq n-a+1>0$. We fill the missing directions with long orthogonal
vectors by setting
\begin{equation}
  \label{eq:completion}
  \mathcal{L}_f
  =\bigl(K_f+M\Z\one_{S_f}\bigr)\oplus M\Z^{n-N}
  \subseteq\Z^n.
\end{equation}
The all-ones direction completes the sum-zero subspace on $S_f$; the
other $n-N$ axes fill the omitted coordinates. Thus $\mathcal{L}_f$ has
rank $n$.

Every vector of $\mathcal{L}_f$ has a unique expression
$z+kM\one_{S_f}+Mw$, with $z\in K_f$, $k\in\Z$, and
$w\in\Z^{T\setminus S_f}$. Orthogonality gives
\begin{equation}
  \label{eq:orthogonality}
  \|z+kM\one_{S_f}+Mw\|^2
  =\|z\|^2+k^2M^2N+M^2\|w\|^2.
\end{equation}
If either $k$ or $w$ is nonzero, then~\eqref{eq:orthogonality} exceeds
$2a$. The added directions therefore leave the minimal vectors unchanged.
Consequently, we have
\begin{equation}
  \label{eq:exact-kissing}
  \rank \mathcal{L}_f=n,
  \qquad
  \min \mathcal{L}_f=2a,
  \qquad
  \tau(\mathcal{L}_f)=R_f;
\end{equation}
combining this with~\eqref{eq:chosen-modulus} proves~\eqref{eq:finite},
establishing Proposition~\ref{prop:finite}.


\subsection*{Optimizing the count}

We now take $n=q$. The scale of the optimal $a$ can be seen before the
asymptotic calculation. Replacing the falling factorial in
$\binom{q}{a}\binom{q-a}{a}$ by $q^{2a}$ gives the rough count
\[
  \frac{q^{2a}}{(a!)^2}\,\frac{1}{q^{a-1}}
  =\frac{q^{a+1}}{(a!)^2};
\]
the ratio of successive values of this expression is $q/(a+1)^2$. The
rough count therefore increases until $a$ is near $\sqrt{q}$ and
decreases afterwards. This identifies the scale; the exact calculation
retains the falling-factorial correction, which supplies the factor
$e^{-2}$.

\begin{proof}[Proof of Theorem~\ref{thm:main}]
Take $n=q$ and $a=\lfloor\sqrt{q}\rfloor$ in
Proposition~\ref{prop:finite}; these parameters are admissible for every
prime power $q\geq5$. Write
\[
  H(q,a)=\frac{\binom{q}{a}\binom{q-a}{a}}{q^{a-1}}
  =\frac{q!}{(a!)^2(q-2a)!q^{a-1}}.
\]
For $a=\sqrt{q}+O(1)$, Stirling's formula and the expansion of the falling
factorial give
\begin{align*}
  \log H(q,a)
  &=\log\frac{q}{2\pi a}
    +2a\left(1+\log\frac{\sqrt{q}}{a}\right)
    -\frac{a(2a-1)}{q}+O(q^{-1/2})\\
  &=2\sqrt{q}+\tfrac{1}{2}\log q-\log(2\pi)-2+o(1).
\end{align*}
Exponentiating the last display and using $e^{o(1)}=1+o(1)$ gives
\[
  H(q,a)=\left(\frac{1}{2\pi e^2}+o(1)\right)
  \sqrt{q}\,e^{2\sqrt{q}}.
\]
Proposition~\ref{prop:finite} gives
\[
  \tau(\mathcal{L}_q)\geq
  \left(1-\frac{1}{q-2a+1}\right)H(q,a).
\]
Since $a=\lfloor\sqrt{q}\rfloor$, the denominator $q-2a+1$ tends to
infinity, so the factor multiplying $H(q,a)$ is $1+o(1)$. Multiplication
by this factor does not change the leading constant in the asymptotic
expression for $H(q,a)$, proving~\eqref{eq:main}.
\end{proof}


\section{Extension to general dimensions}
\label{sec:other-dimensions}

We now extend the construction to general dimensions by way of nearby
prime powers. For a given dimension $n$, we choose a nearby prime power
$q$. If $q\leq n$, we adjoin long orthogonal coordinate directions; if
$q>n$, we apply Proposition~\ref{prop:finite} to $n$ labels in $\F_q$.
The loss in the kissing-number bound is controlled by $|q-n|$. Put
\[
  \Delta(n)=\min\{|n-q|:\text{$q\geq5$ is a prime power}\};
\]
any bound on the distance to a prime also bounds $\Delta(n)$.

\begin{proposition}
\label{prop:dimension-transfer}
Let $\mathcal{N}$ be an unbounded set of positive integers, and suppose
that $\Delta(n)\leq D(n)<n$ for $n\in\mathcal{N}$, with
$n-D(n)\to\infty$.

There are full-rank lattices $\mathcal{L}_n\subseteq\Z^n$ such that, as
$n\to\infty$ through $\mathcal{N}$,
\begin{equation}
  \label{eq:dimension-transfer}
  \tau(\mathcal{L}_n)\geq
  \left(\frac{1}{2\pi e^2}+o(1)\right)
  \sqrt{n-D(n)}\,\exp\bigl(2\sqrt{n-D(n)}\bigr).
\end{equation}
In particular, if $D(n)=o(n^{3/4})$, then
\begin{equation}
  \label{eq:distance-loss}
  \tau(\mathcal{L}_n)\geq
  \left(\frac{1}{2\pi e^2}+o(1)\right)\sqrt{n}\,
  \exp\left(2\sqrt{n}-\frac{D(n)}{\sqrt{n}}\right).
\end{equation}
\end{proposition}

\begin{proof}
Choose a prime power $q\geq5$ with $|q-n|\leq D(n)$, and put $r=n-D(n)$.
If $q\leq n$, then adjoin sufficiently long integer coordinate axes to
the lattice of Theorem~\ref{thm:main}---this preserves its kissing number.
Since $q\geq r\to\infty$ and
$\sqrt{q}\,e^{2\sqrt{q}}\geq\sqrt{r}\,e^{2\sqrt{r}}$, we obtain
\eqref{eq:dimension-transfer}.

If $q>n$, then we apply Proposition~\ref{prop:finite} to $n$ labels in
$\F_q$ and $a=\lfloor\sqrt{n}\rfloor$. The Stirling calculation in the
proof of Theorem~\ref{thm:main} applies to every integer $n$ and gives
\[
  \tau(\mathcal{L}_n)\geq
  \left(\frac{1}{2\pi e^2}+o(1)\right)
  \sqrt{n}\,e^{2\sqrt{n}}\left(\frac{n}{q}\right)^{a-1}.
\]
Because $a-1\leq\sqrt{n}$ and $\log(1+t)\leq t$ for $t\geq0$, we have
\[
  \left(\frac{n}{q}\right)^{a-1}
  \geq\exp\left(-\frac{q-n}{\sqrt{n}}\right)
  \geq\exp\left(-\frac{D(n)}{\sqrt{n}}\right);
\]
together with $\sqrt{n}\geq\sqrt{r}$ and
$2\sqrt{n}-D(n)/\sqrt{n}\geq2\sqrt{r}$, this proves
\eqref{eq:dimension-transfer} in the second case.

Finally, if $D(n)=o(n^{3/4})$, then $\sqrt{n-D(n)}\sim\sqrt{n}$ and
\[
  2\sqrt{n-D(n)}
  =2\sqrt{n}-\frac{D(n)}{\sqrt{n}}
  +O\left(\frac{D(n)^2}{n^{3/2}}\right).
\]
The assumption $D(n)=o(n^{3/4})$ is used only to make the error $o(1)$,
which gives~\eqref{eq:distance-loss}.
\end{proof}

The bounds in Proposition~\ref{prop:dimension-transfer} distinguish three
scales. A distance $D(n)=o(n)$ preserves the exponential coefficient $2$,
giving $\tau(\mathcal{L}_n)\geq\exp((2-o(1))\sqrt{n})$.
A distance $D(n)=A\sqrt{n}+o(\sqrt{n})$, for fixed $A\geq0$, gives the
leading constant $e^{-A}/(2\pi e^2)$ in~\eqref{eq:main};
$D(n)=o(\sqrt{n})$ preserves the full asymptotic bound.

As the construction accepts a prime power on either side of $n$, we can
center a one-sided prime interval at $n$ and halve the distance bound, to
obtain the following direct substitution rule.

\begin{corollary}
\label{cor:short-intervals}
Suppose that, for fixed $C>0$ and $0<\theta<3/4$, every sufficiently large
interval $[x-Cx^\theta,x]$ contains a prime power. Then there are full-rank
lattices $\mathcal{L}_n\subseteq\Z^n$ with
\begin{equation}
  \label{eq:interval-transfer}
  \tau(\mathcal{L}_n)\geq
  \left(\frac{1}{2\pi e^2}+o(1)\right)\sqrt{n}\,
  \exp\left(2\sqrt{n}-\frac{C}{2}n^{\theta-1/2}\right)
  \qquad(n\to\infty).
\end{equation}
\end{corollary}

\begin{proof}
Take $x=n+\tfrac{C}{2}n^\theta$. Since
$Cx^\theta=Cn^\theta+O(n^{2\theta-1})$, the interval hypothesis gives
\[
  \Delta(n)\leq\frac{C}{2}n^\theta+O(n^{2\theta-1}).
\]
Then we apply~\eqref{eq:distance-loss}; the error divided by $\sqrt{n}$ is
$O(n^{2\theta-3/2})=o(1)$.
\end{proof}

Thus any improvement in $\theta$ improves the loss exponent $\theta-1/2$
in~\eqref{eq:interval-transfer}, and improving $C$ improves the
coefficient. An interval exponent $\theta<1/2$ would recover the full
bound~\eqref{eq:main} in every sufficiently large dimension.


\subsection*{Bounds for every sufficiently large dimension}

Baker, Harman, and Pintz~\cite{BHP} prove the interval hypothesis with
$C=1$ and $\theta=21/40$. Corollary~\ref{cor:short-intervals} thus gives
\begin{equation}
  \label{eq:all-dimensions}
  \tau(\mathcal{L}_n)\geq
  \left(\frac{1}{2\pi e^2}+o(1)\right)
  \sqrt{n}\,\exp\bigl(2\sqrt{n}-\tfrac{1}{2}n^{1/40}\bigr).
\end{equation}


\subsection*{A set of dimensions of natural density $1$}

Suppose that, for some fixed $0<\eta<1/2$, all but $o(X)$ integers
$n\in[X,2X]$ have a prime within distance $O(n^\eta)$.
Dyadic summation then gives $\Delta(n)=o(\sqrt{n})$ on a set
$\mathcal{N}$ of natural density $1$, so
Proposition~\ref{prop:dimension-transfer} yields
\begin{equation}
  \label{eq:density-one}
  \tau(\mathcal{L}_n)\geq
  \left(\frac{1}{2\pi e^2}+o(1)\right)\sqrt{n}\,e^{2\sqrt{n}}
  \qquad(n\to\infty,\ n\in\mathcal{N}).
\end{equation}

Available interval exponents include $\eta=2/15+\varepsilon$ from Guth
and Maynard~\cite[Corollary~1.4]{GuthMaynard2026} and
$\eta=1/20+\varepsilon$ from Jia~\cite{JiaShortIntervals}. Both give the
same leading asymptotic~\eqref{eq:density-one}, since
$D(n)/\sqrt{n}\to0$.


\section{Remarks}
\label{sec:remarks}


\subsection*{Relation to Craig's lattices}

For a prime $p$ and $2\leq a<p/2$, Craig's lattice has the moment
description
\begin{equation}
  \label{eq:craig}
  A_{p-1}^{(a)}
  =\left\{
    x\in\Z^p:
    \sum_{s\in\F_p}x_s=0,\quad
    \sum_{s\in\F_p}s^j x_s\equiv0\!\!\!\!\pmod{p}
    \ \text{for $1\leq j\leq a-1$}
  \right\};
\end{equation}
see~\cite{Craig,ConwaySloane} and~\cite[Section~8.1]{Bacher}.
The moment conditions are equivalent to
\[
  \prod_{s\in\F_p}(1-sT)^{x_s}\equiv1\!\!\!\!\pmod{T^a}.
\]
Thus Craig's bound $\min A_{p-1}^{(a)}\geq2a$ follows from the same
degree argument as Lemma~\ref{lem:minimum}, with a congruence at infinity
replacing the finite polynomial modulus.
Bennett and Peikert~\cite{BennettPeikert} also note the close relationship
between their lattices and Craig's.

The vectors of squared norm $2a$ in Craig's lattice are exactly
$\one_A-\one_B$ for disjoint $a$-element sets $A,B\subseteq\F_p$
satisfying $P_A-P_B\in\F_p^\times$.
These constant-difference pairs form a subset of the pairs excluded by
our condition $\deg(P_A-P_B)=a-1$.
The key distinction, however, is in the counting argument.
Averaging $a$-element sets over the $p^{a-1}$ moment
classes~\cite[Corollary~8.3]{Bacher} does not force even two sets into one
class when $a\sim\sqrt{p}$, since the average class size satisfies
\[
  \frac{\binom{p}{a}}{p^{a-1}}\leq\frac{p}{a!}<1
\]
for sufficiently large $p$.

Although Bacher's bound yields superpolynomial kissing numbers when $a$
grows sufficiently slowly with $p$, it becomes ineffective in the regime
$a\sim\sqrt{p}$ relevant here.
Our argument instead assigns disjoint pairs $(A,B)$ to moduli $f_{A,B}$,
obtaining an unconditional lower bound for the number of minimal vectors
in a suitably chosen $K_f$.

An earlier version of this paper~\cite{Laarhoven}, following Graham and
Sloane~\cite{GrahamSloane} and Bennett and
Peikert~\cite[Section~3.2]{BennettPeikert}, selected a large moment class
of binary vectors of weight $2a$ and adjoined its coset to Craig's
lattice.
The class average was $\binom{p}{2a}/p^{a-1}$; optimizing gave
$2a\sim\sqrt{p}$ and a kissing-number lower bound
$\exp((1+o(1))\sqrt{p})$.
Here each support of size $2a$ instead supplies $\binom{2a}{a}$ signed
candidates in the sum-zero hyperplane.
For $q=p$, our lower bound is the binary-class average multiplied by
\[
  \left(1-\frac{1}{p-2a+1}\right)\binom{2a}{a},
\]
where the first factor accounts for the degree condition.
Optimizing now gives $a\sim\sqrt{p}$ and increases the square-root
exponential coefficient from $1$ to $2$.


\subsection*{Computing a modulus}

The modulus in~\eqref{eq:chosen-modulus} can be found by enumerating the
good pairs and counting the fibers of~\eqref{eq:assignment-map}. This
gives a finite search, but we do not obtain a polynomial-time method for
selecting and certifying a modulus attaining~\eqref{eq:finite}.
Computing a basis is a separate task.


\subsection*{Exceptional moduli}

For fixed $c>0$ and $a=c\sqrt{q}+O(1)$, the count in the proof satisfies
\[
  \log H(q,a)=2c(1-\log c)\sqrt{q}+O(\log q).
\]
The coefficient is maximized at $c=1$, with value $2$. More generally,
$a!\geq(a/e)^a$ gives
$H(q,a)\leq q(e\sqrt{q}/a)^{2a}\leq qe^{2\sqrt{q}}$ for every admissible
$a$, so optimizing $a$ cannot increase the coefficient supplied by this
averaging bound.

The individual counts $R_f$ may nevertheless exceed their average. If one
could identify a family of moduli with at least $\exp(\delta\sqrt{q})$
times the average pair count, for fixed $\delta>0$ and
$a=\sqrt{q}+O(1)$, that would increase the square-root exponential
coefficient from $2$ to $2+\delta$.


\providecommand{\bysame}{\leavevmode\hbox to3em{\hrulefill}\thinspace}
\providecommand{\MR}{\relax\ifhmode\unskip\space\fi MR }
\providecommand{\MRhref}[2]{%
  \href{http://www.ams.org/mathscinet-getitem?mr=#1}{#2}
}
\providecommand{\href}[2]{#2}

\end{document}